\documentclass[a4paper,11pt]{article}
\usepackage{amsfonts}
\usepackage{tikz}
\usetikzlibrary{matrix,arrows}

\usepackage{latexsym}
\usepackage{amsthm}
\usepackage{amsmath}
\usepackage{graphicx}

\newtheorem{Le}{Lemma}

\newtheorem{The}{Theorem}

\newtheorem{Pro}{Proposition}
\theoremstyle{definition}
\newtheorem{Fact}{Fact}

\newcommand{\vinc}[3]{
\begin{tikzpicture}[baseline = (X.base)]
	\useasboundingbox (0.1,0) rectangle (#1*0.23,0.1);
	\foreach \x/\y in {#2}
	{
		\draw (\x*0.2,0) node (X) {$\y$};
	}
	
	\foreach \z in {#3}
	{
		\ifnum 0<\z
			\ifnum \z<#1
				\draw[thick] (\z*0.2-0.07,-0.19) -- (\z*0.2+0.27,-0.19);
			\fi
		\fi
		
		\ifnum 0=\z
			\draw[thick] (0.07,0.1) -- (0.07,-0.19) -- (0.21,-0.19);
		\fi
		
		\ifnum \z=#1
			\draw[thick] (\z*0.2+0.14,0.1) -- (\z*0.2+0.14,-0.19) -- (\z*0.2,-0.19);
		\fi
	}
\end{tikzpicture}
}
\def\bSac{\vinc{3}{1/2,2/1,3/3}{2}}
\def\baSc{{\vinc{3}{1/2,2/1,3/3}{1}}}
\def\bSca{\vinc{3}{1/2,2/3,3/1}{2}}
\def\bcSa{\vinc{3}{1/2,2/3,3/1}{1}}
\def\cSab{\vinc{3}{1/3,2/1,3/2}{2}}

\def\aScb{\vinc{3}{1/1,2/3,3/2}{2}}

\def\caSb{\vinc{3}{1/3,2/1,3/2}{1}}

\def\aScb{\vinc{3}{1/1,2/3,3/2}{2}}

\def\acSb{\vinc{3}{1/1,2/3,3/2}{1}}

\def\ba{\vinc{2}{1/2,2/1}{1}}
\def\abF{\vinc{2}{1/1,2/2}{2}\ }

\def\baF{\vinc{2}{1/2,2/1}{2}\ }

\def\mn{\mbox{-}}
\def\des{\mathsf{des}}
\def\rlm{\mathsf{rlmin}}
\def\rlM{\mathsf{rlmax}}

\def\st{{\scriptstyle \mathsf{ST}}}
\def\stp{{\scriptstyle \mathsf{ST'}}}

\title{The equidistribution of some length three 
vincular patterns on~$S_n(132)$} 
\author{{\sc Vincent}
}

\author{
Vincent {\sc Vajnovszki}\\ 
{\small LE2I, Universit\'e de Bourgogne}\\
{\small BP 47870, 21078 Dijon Cedex, France}\\
{\small \tt vvajnov@u-bourgogne.fr}
}
\begin{document}

\maketitle

\begin{abstract}

In 2012 B\'ona showed the rather surprising fact that the cumulative
number of occurrences of the classical patterns
$231$ and $213$ are the same on the set of permutations 
avoiding $132$, beside the pattern based statistics 
$231$ and  $213$ do not have the same distribution on this set.
Here we show that if it is required for the symbols playing the role of $1$ and $3$ 
in the occurrences of $231$ and $213$
to be adjacent, then the obtained statistics are
equidistributed on the set of $132$-avoiding
permutations.
Actually, expressed in terms of vincular patterns, 
we prove the following more general results: 
the statistics based on the patterns   
$\bSca$, 
$\bSac$ and 
$\baSc$,
together with other statistics,  have the same joint distribution on $S_n(132)$,
and so do the patterns $\bcSa$ and $\cSab$; 
and up to trivial transformations, these statistics 
are the only based on length three proper (not classical nor adjacent) vincular patterns 
which are equidistributed on a set of permutations avoiding a classical length three pattern.
\end{abstract}

\section{Introduction}

In \cite{BBS_10} Barnabei, Bonetti and Silimbani 
showed the equidistribution of some length three consecutive patterns involvement
statistics on the set of permutations avoiding the classical pattern 
$312$ (or equivalently, $132$).
And in \cite{Bona_13} B\'ona showed the surprising fact that the total
number of occurrences of the patterns $231$ and $213$ is the same on 
the set of $132$-avoiding permutations, beside the pattern based statistics 
$231$ and $213$ are not equidistributed on this set.
In \cite{Hom_12}, Homberger, generalizing B\'ona's result, gave the total number of 
occurrences of each classical length three pattern on the set of $123$-avoiding permutations, and 
showed that the total number of occurrences of the pattern 
$231$ is the same in the set of $123$- and $132$-avoiding
permutations, despite the pattern statistic $231$ has different distribution on 
the two sets.
Motivated by these, Burnstein and Elizalde gave in \cite{Bur_Eliz_13}, in a much more
general context, the total number of occurrences of any vincular pattern 
of length three on $231$-avoiding (or equivalently, $132$-avoiding) permutations.

In this paper we show that, on the set of $132$-avoiding permutations, 
the vincular pattern based statistics $\bSca$, $\bSac$ and $\baSc$
are equidistributed, and so are  $\bcSa$ and $\cSab$; 
and numerical evidences show that, up to trivial transformations,
these patterns are the only length three proper (not classical nor adjacent)
vincular patterns equidistributed 
on a set of permutations avoiding a classical length three pattern.

It is worth to mention that, on the set of unrestricted permutations,
the statistics $\bcSa$ and $\cSab$ are trivially equidistributed, and 
so are $\bSca$ and $\bSac$ (which is all but obvious on $132$-avoiding permutations),
and this last distribution is different from that of $\baSc$.

More precisely, we show bijectively the equidistribution on $132$-avoiding
permutations of the multistatistics
\begin{itemize}
\item $(\bSca,\bSac,\rlm,\rlM)$
      and $(\bSac,\bSca,\rlM,\rlm)$, 
\item $(\bSca,\des)$  and $(\bSac,\des)$, 
\item $(\baSc,\des,\abF)$ and $(\bSac,\des,\abF)$,
\item $(\bcSa,\cSab,\des)$ and $(\cSab,\bcSa,\des)$,
\end{itemize}
where 
$\rlM,\rlm$ and $\des$ are respectively, the number of 
right-to-left maxima, right-to-left minima and descents.
The corresponding bijections (the last of them being straightforward)
are presented in Subsection 3.

\section{Notations and definitions}
\label{Notations}

A {\it permutation} of length $n$ is a bijection from the set $\{1,2,\ldots, n\}$
to itself and we write permutations in {\it one-line notation}, that is,
as words $\pi=\pi_1\pi_2\ldots \pi_n$, where $\pi_i$ is the image of
$i$ under $\pi$. We let $S_n$ denote the set of permutations of
length $n\geq 0$, and $S=\cup_{n\geq 0}S_n$. 

\subsection*{Vincular patterns}
\label{vinc-pat}

Let $\sigma\in S_k$ and  $\pi=\pi_1\pi_2\ldots \pi_n\in S_n$, 
$k\leq n$, be two permutations.
One says that  $\sigma$ occurs as a (classical) pattern in $\pi$
if there is a sequence $1\leq i_1<i_2<\cdots<i_k\leq n$ such that
$\pi_{i_1}\pi_{i_2}\cdots \pi_{i_k}$ is order-isomorphic with
$\pi$. For example, $231$ occurs as a pattern in $13452$,
and the three occurrences of it are $342$, $352$ and $452$.

Vincular patterns were introduced in \cite{BabSteim}
and they were extensively studied since then (see 
Chapter 7 in \cite{Kit} for a comprehensive description of results on these patterns). 
Vincular patterns 
generalize classical patterns and they are defined as follows:
\begin{itemize}
\item Any pair of two adjacent letters may now be underlined, which means that the 
corresponding letters in the permutation must be 
adjacent.
(The original notation for vincular patterns uses dashes: 
the absence of a dash between two letters of a pattern means that these letters 
are adjacent in the permutation.)  For example, the pattern 
$\vinc{3}{1/2,2/1,3/3}{2}$ occurs in the permutation 425163 four times, namely, as 
the subsequences $425$, $416$, $216$ and $516$. Note that, the subsequences $426$ and $213$ are {\em not} 
occurrences of the pattern because their last two letters are not adjacent in the permutation. 
\item If a pattern begins (resp., ends) with a 
hook then its occurrence is required to begin (resp., end) with the leftmost (resp., rightmost) 
letter in the permutation. 
(In the original notation the role of hooks was played by square brackets.) 
For example, there are two occurrences of the pattern 
$\vinc{3}{1/2,2/1,3/3}{0,2}$ in the permutation $425163$, which are the subsequences 
$425$ and $416$.
\end{itemize}


\subsection*{Statistics}

A {\em statistic} on a set of permutations is simply a function from the set 
to $\mathbb{N}$, and a {\it multistatistic} is a tuple of statistics.
A classical example of statistic on $S_n$ is the descent number
$$
\displaystyle 
\des\, \pi = \text{card}\, \{i\ :\ 1\leq i<n, \pi_i>\pi_{i+1}\},
$$
for example $\des\,45312=2$.

In a permutation $\pi=\pi_1\pi_2\ldots \pi_n$, $\pi_i$ is a {\it right-to-left maximum} 
if $\pi_i>\pi_j$ for all $j>i$;
and the number of right-to-left maxima of $\pi$ is denoted by 
$\rlM\, \pi$.
Similarly, $\pi_i$ is a  {\it right-to-left minimum} if
$\pi_i<\pi_j$ for all $j>i$; 
and the number of right-to-left minima of $\pi$ is denoted by $\rlm\,\pi$.
Both, $\rlM$ and $\rlm$ are statistics on $S_n$.

For a set of permutations $S$, two statistics $\st$ and $\st'$ have the same distribution 
(or are equidistributed) on $S$
if, for any $k$, 
$$
\mathrm{card}\{\pi\in S: \st\,\pi=k\}=
\mathrm{card}\{\pi\in S: \stp\,\pi=k\},
$$
and the multistatistics $(\st_1,\st_2,\ldots,\st_p)$ and
$(\st_1',\st_2',\ldots,\st_p')$ have the same distribution if, for any 
$p$-tuple $k=(k_1,k_2,\ldots, k_p)$,
$$
\mathrm{card}\{\pi\in S: (\st_1,\st_2,\ldots,\st_p)\,\pi=k\}=
\mathrm{card}\{\pi\in S: (\st_1',\st_2',\ldots,\st_p')\,\pi=k \}.
$$

For a permutation $\pi$ and a (vincular) patterns $\sigma$
we denote by $(\sigma)\,\pi$ the number of occurrences of  
this pattern in $\pi$, and $(\sigma)$ becomes 
a permutation statistic. For example,
$(\ba)\,\pi$ is $\des\,\pi$; 
$(21)\,\pi$ is the inversion number of $\pi$; and
$(\abF)\pi$ is the last value of $\pi$ minus one.
Similarly, for a set of (vincular) patterns $\{\sigma,\tau,\ldots \}$,
we denote by $(\sigma+\tau+\cdots)\,\pi$ the number of occurrences of  
these patterns in $\pi$.

\subsection*{Sum decomposition}

For a permutation $\pi$, $|\pi|$ denotes its length (and so, $\pi\in S_{|\pi|}$),
and for two permutations $\alpha$ and~$\beta$
\begin{itemize}
\item the {\it skew sum} of $\alpha$ and $\beta$, denoted $\alpha\ominus\beta$,
is the permutation $\pi$ of length  $|\alpha|+|\beta|$ with
$$
\pi_i=
\left\{ \begin {array}{ccl}
\alpha_i+|\beta| & {\rm if} & 1\leq i\leq |\alpha|, \\
\beta_{i-|\alpha|} & {\rm if} & |\alpha|+1\leq i\leq |\alpha|+|\beta|,
\end {array}
\right.
$$
and

\item the {\it direct sum} of $\alpha$ and $\beta$, denoted $\alpha\oplus\beta$,
is the permutation $\pi$ of length  $|\alpha|+|\beta|$ with
$$
\pi_i=
\left\{ \begin {array}{ccl}
\alpha_i & {\rm if} & 1\leq i\leq |\alpha|, \\
\beta_{i-|\alpha|}+|\alpha| & {\rm if} & |\alpha|+1\leq i\leq |\alpha|+|\beta|.
\end {array}
\right.
$$
\end{itemize}
It is easy to check the following.

\begin{Fact}
\label{fact_des}
For two permutations $\alpha$ and $\beta$,  
$\des\, \alpha\oplus\beta=\des\, \alpha+\des\, \beta$ and, when $\alpha$ is not
empty, $\des\, \alpha\ominus\beta=\des\, \alpha+\des\, \beta +1$. 
\end{Fact}

The following characterization of $132$-avoiding permutations is folklore.

\begin{Fact}
\label{fact_decomposition}
For a non-empty permutation $\pi\in S_n$, $n\geq 1$, the following are equivalent:
\begin{itemize}
\item $\pi$ avoids $132$;
\item $\pi$ can uniquely be written as $(\alpha\oplus 1)\ominus \beta$,
      where  $\alpha$ and $\beta$ are (possibly empty)  $132$-avoiding permutations;
\item $\pi$ can uniquely be written as $\alpha\ominus(\beta\oplus 1)$,
      where  $\alpha$ and $\beta$ are (possibly empty) $132$-avoiding permutations.
\end{itemize}
\end{Fact}
\noindent
See Figure \ref{Deux_tr}.

\begin{figure}
\begin{center}
\begin{tabular}{ccc}
   \unitlength=2mm
\begin{picture}(9,9)
\put(0,0){\line(1,0){8}}
\put(0,0){\line(0,1){8}}
\put(0,8){\line(1,0){8}}
\put(8,0){\line(0,1){8}}
%
\put(3.28,7.51){\circle*{0.7}} 
\put(0,7){\line(1,0){2.65}}
\put(0,4.45){\line(1,0){2.65}}
\put(2.7,4.42){\line(0,1){2.61}}
\put(4.1,0){\line(0,1){3.95}}
\put(4.1,3.9){\line(1,0){3.95}}
\put(3,8.7){$\scriptstyle n$}
\put(0.9,5.4){$\scriptstyle \alpha$}
\put(5.4,1.4){$\scriptstyle \beta$}
\end{picture} 
 & &
   \unitlength=2mm
\begin{picture}(9,9)
\put(0,0){\line(1,0){8.}}
\put(0,0){\line(0,1){8.}}
\put(0,8){\line(1,0){8.0}}
\put(8,0){\line(0,1){8.0}}
%
\put(7.51,3.28){\circle*{0.7}} 
\put(6.9,0){\line(0,1){2.65}}
\put(4.35,0){\line(0,1){2.65}}
\put(4.3,2.65){\line(1,0){2.65}}
\put(4.,4.17){\line(0,1){3.8}}
\put(0,4.2){\line(1,0){3.95}}
\put(8.5,3){$\scriptstyle \pi_n$}
\put(1.5,5.7){$\scriptstyle \alpha$}
\put(5.2,0.9){$\scriptstyle \beta$}

\end{picture} 
\\
(1)  & & (2)  \\
\end{tabular}
\caption{\label{Deux_tr} The decomposition: (1)
$\pi=(\alpha\oplus 1)\ominus \beta$, 
and (2)  $\pi=\alpha\ominus(\beta\oplus 1)$ of $\pi\in S_n(132)$, $n\geq 1$.
}
\end{center}
\end{figure}
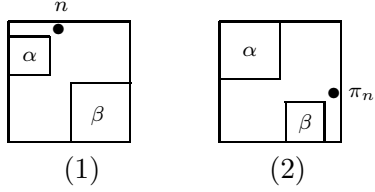

\subsection*{Permutation symmetries}

For a permutation $\pi\in S_n$, the {\it reverse} and 
{\it complement} of $\pi$, denoted $\pi^r$ and $\pi^c$ respectively, 
are defined as:
\begin{itemize}
\item $\pi^r_i=\pi_{n-i+1}$,
\item $\pi^c_i=n-\pi_{i}+1$,
\end{itemize}
and both operations can naturally be extended to vincular patters; for
instance, the reverse of 
$\bcSa$ is $\aScb$, and the complement of 
$\bcSa$ is $\baSc$. These operations preserve pattern containment, in the sense that, 
if the (vincular) patter $\sigma$ is contained in the permutation $\pi$, then
$\sigma^r$ is contained in $\pi^r$, and $\sigma^c$ is contained in $\pi^c$.

The {\it inverse} of $\pi$, denoted $\pi^{-1}$, is defined as: 
\begin{itemize}
\item $\pi^{-1}_{\pi_i}=i$,
\end{itemize}
but, unlike the reverse and complement, it can not be extended to vincular patters:
in general, the inverse of a vincular pattern is a {\it bivincular pattern}
(see for example \cite[p. 13]{Kit} for its formal definition), we will not
consider here. 

\section{Main results}
\label{main_results}

\noindent
{\bf 3.1 Equidistribution of 
         $(\bSca,\bSac,\rlM,\rlm)$ and $(\bSac,\bSca,\rlm,\rlM)$ on 
         $S_n(132)$: bijection~$\phi$
}

\medskip


\noindent
We define a mapping $\phi$ on $S_n(132)$ and we will see that it is an involution, 
that is, a bijection from $S_n(132)$ into itself, which is its own inverse;
and Theorem \ref{The_1} below shows the desired equidistribution.

\medskip

The mapping $\phi$ is recursively defined as:
if $\pi$ is the empty permutation (that is, $n=0$), then $\phi(\pi)=\pi$; 
and if $\pi\in S_n(132)$, $n\geq 1$, with $\pi=\alpha \ominus (\beta \oplus 1)$
for some $132$-voiding permutations $\alpha$ and $\beta$,
then  
$$\phi(\pi)=\phi(\beta)\ominus (\phi(\alpha)\oplus 1).
$$
See Figure \ref{cons_phi} for this definition.

Note that $\phi$ is in some sense similar with the inversion $^{-1}$,
but is fundamentally different from it. Indeed, the inversion
when restricted to $S_n(132)$ satisfies: 
$\left(\alpha \ominus (\beta \oplus 1)\right)^{-1}=
(\beta^{-1}\oplus 1) \ominus \alpha^{-1}$, see Figure \ref{cons_inv}.       

\begin{figure}[h]
\begin{center}
\begin{tabular}{c}
\unitlength=4mm
\begin{picture}(0.3,5)
\put(-0.7,2.3){$\scriptstyle\pi=$}
\end{picture}
\unitlength=1.2mm
\begin{picture}(19,19)
\put(0,0){\line(1,0){19}}
\put(0,0){\line(0,1){19}}
\put(19,0){\line(0,1){19}}
\put(0,19){\line(1,0){19}}

\put(0,13){\line(1,0){6}}
\put(6,13){\line(0,1){6}}

\put(6,0){\line(0,1){11}}
\put(6,11){\line(1,0){11}}
\put(17,0){\line(0,1){11}}

\put(2.2,15.2){$\scriptstyle\alpha$}
\put(10.4,5.){$\scriptstyle\beta$}
\put(18,12){\circle*{1.2}} 
\end{picture} 
\unitlength=4mm
\begin{picture}(1,5)
\put(0.,2.3){$\rightarrow$}
\end{picture}
\unitlength=1.2mm
\begin{picture}(19,19)
\put(0,0){\line(1,0){19}}
\put(0,0){\line(0,1){19}}
\put(19,0){\line(0,1){19}}
\put(0,19){\line(1,0){19}}

\put(0,8){\line(1,0){11}}
\put(11,8){\line(0,1){11}}

\put(11,0){\line(0,1){6}}
\put(17,0){\line(0,1){6}}
\put(11,6){\line(1,0){6}}

\put(2.7,12.8){$\scriptstyle\phi(\beta)$}
\put(11.5,2.7){$\scriptstyle\phi(\alpha)$}
\put(18,7){\circle*{1.2}} 
\end{picture} 
\unitlength=4mm
\begin{picture}(0.3,5)
\put(0,2.3){$\scriptstyle=\phi(\pi)$}
\end{picture}
\end{tabular}
\end{center}
\caption{
\label{cons_phi}
The recursive definition of $\phi(\pi)$.}
\end{figure}
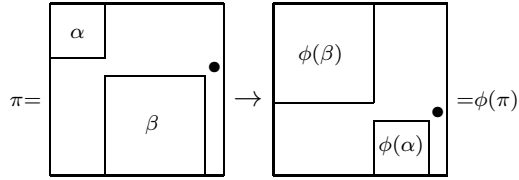

By induction on $n$, from the definition of $\phi$ it follows
that if $\pi\in S_n(132)$, then $\phi(\pi)\in S_n(132)$ and $\phi(\phi(\pi))=\pi$, 
and in particular $\phi$ is a bijection on $S_n(132)$ to itself.

In the proof of the next theorem we will need the following result.

\begin{Pro}
\label{one_pro}
For $\pi\in S_n(132)$ we have 
$(\abF)\, \pi=(\baF)\, \phi(\pi)$ and
$(\baF)\, \pi=(\abF)\, \phi(\pi)$.
\end{Pro}
\begin{proof}
If $\pi=\alpha \ominus (\beta \oplus 1)$ is a non-empty $132$-avoiding permutation, 
then we have 
\begin{itemize}
\item $(\baF)\, \phi(\pi)=|\phi(\beta)|=|\beta|=(\abF)\, \pi$; and 
\item $(\abF)\, \phi(\pi)=|\phi(\alpha)|=|\alpha|=(\baF)\, \pi$.
\end{itemize}
\end{proof}

\begin{The}
\label{The_1}
If $\pi\in S_n(132)$, then
$$
(\bSac,\bSca,\rlm,\rlM)\,\phi(\pi)=(\bSca,\bSac,\rlM,\rlm)\,\pi.
$$
\end{The}
\begin{proof} By induction on $n$.
Trivially, the statement holds for $n=0$, and let suppose that it is true for any $n<m$, 
and consider $\pi=\alpha \ominus (\beta \oplus 1)\in S_m(132)$ for
some $132$-avoiding permutations $\alpha$ and $\beta$. 

First we prove that $\rlM\,\pi=\rlm\,\phi(\pi)$:
\begin{equation*}
\begin{array}{rclr}
\rlm\, \phi(\pi) & = & 1+ \rlm\, \phi(\alpha) \\
	         & = & 1+ \rlM\, \alpha   & \mbox{(by the induction hypothesis)}\\
                 & = & \rlM\, \pi. 
\end{array}
\end{equation*}

\medskip

An occurrence of $\bSac$ in $\phi(\pi)$ can be found  either 
in $\phi(\alpha)$, or in $\phi(\beta)$, or has the form $abc$ with 
$ab$ an occurrence of $\baF$ in $\phi(\alpha)$ and $c$ the last symbol of $\phi(\pi)$. 
Thus
\begin{equation*}
\begin{array}{rclr}
(\bSac)\, \phi(\pi) & = & (\bSac)\, \phi(\alpha) + (\bSac)\, \phi(\beta) + (\baF)\, \phi(\alpha) \\
	            & = & (\bSca)\, \alpha + (\bSca)\, \beta + (\abF)\, \alpha  & \mbox{(by the induction hypothesis}\\
	            &   &   & \mbox{ and Proposition \ref{one_pro})}.
\end{array}
\end{equation*}
And an occurrence of  $\bSca$ in $\pi$ can be found either 
in $\alpha$, or in $\beta$, or has the form $abc$ with 
$ab$ an occurrence of $\abF$ in $\alpha$ and $c$ the first symbol of
$\beta$ if it is not empty, and the last symbol of $\pi$ when $\beta$ is empty. 
Thus
\begin{equation*}
\begin{array}{rclr}
(\bSca)\, \pi & = & (\bSca)\, \alpha + (\bSca)\, \beta + (\abF)\, \alpha,  & 
\end{array}
\end{equation*}
and finally $(\bSac)\, \phi(\pi)=(\bSca)\, \pi$.

Moreover, since $\phi$ is an involution, it follows that  
$(\bSca)\, \phi(\pi)=(\bSac)\, \pi$ and $\rlm\, \pi=\rlM\, \phi(\pi)$.
\end{proof}

The three-statistics $(\bSca,\bSac,\des)$ and $(\bSac,\bSca,\des)$  
do not have the same distribution on $S_n(132)$; however, 
Theorem \ref{The_psi} below says that separetely $(\bSca)$ and $(\bSac)$,
together with~$\des$, have the same joint distribution.

\medskip

\noindent
{\bf 3.2 Equidistribution of $(\bSca,\des)$ and 
      $(\bSac,\des)$ on $S_n(132)$: bijection~$\psi$ 
}

\medskip


\noindent
Now, we define the mapping $\psi:S_n(132)\to S_n$ 
by $\psi(1)=1$ if $n=1$, and for $n\geq 2$, 
$\psi(\pi)$ is defined recursively below, according to 
three cases: 
$\pi^{-1}_n=1$, 
$2\leq \pi^{-1}_n\leq n-1$, and 
$\pi^{-1}_n=n$; see Figure \ref{three_cases}.

\medskip

Let $\pi\in S_n(132)$, $n\geq 2$.
\begin{enumerate}
\item If $\pi$ has the form $1\ominus\alpha$ (or equivalently, $\pi^{-1}_n=1$), then
      $\psi(\pi)$ is simply $1\ominus\psi(\alpha)$. 

\item If $\pi$ has the form $(\alpha \oplus 1 ) \ominus \beta$ for some non-empty permutations 
      $\alpha$ and $\beta$ (or equivalently, $2\leq \pi^{-1}_n\leq n-1$), then $\psi(\pi)$ is obtained by:
      \begin{enumerate}
      \item considering the (possibly empty) permutations $\gamma $ and $\delta$ 
                 with $\psi(\beta)=\gamma\ominus(\delta\oplus 1)$, and  
      
      \item defining $\psi(\pi)$ as 
                  $((\psi(\alpha)\ominus (\delta \oplus 1) )\oplus 1)\ominus \gamma$.
      \end{enumerate}
      (Note that $\alpha$, $\beta$, $\gamma$ and $\delta$ are
      $132$-avoiding permutations.)
\item  If $\pi$ has the form $\alpha \oplus 1$ for some non-empty permutation 
      $\alpha$ (or equivalently, $\pi^{-1}_n=n$), then $\psi(\pi)$ is obtained by:    
      \begin{enumerate}
      \item considering the (possibly empty) permutations $\gamma $ and $\delta$ 
                 with $\psi(\alpha)=(\gamma\oplus 1) \ominus \delta$, and
      \item defining $\psi(\pi)$ as $((\gamma\oplus 1)\oplus 1)\ominus \delta$.
      \end{enumerate} 
      (Note that $\alpha$, $\gamma$ and $\delta$ are $132$-avoiding
      permutations.)
\end{enumerate}

\noindent
From the above definition of $\psi$ it is easy to check the following.

\begin{Pro} 
\label{Pro_psi}
Let $\pi\in S_n(132)$ and $\sigma=\psi(\pi)$.
\begin{enumerate}
\item $\pi^{-1}_n=1$ iff $\sigma^{-1}_n=1$,
\item $2\leq \pi^{-1}_n\leq n-1$ iff  $\sigma^{-1}_n>1$ and  $\sigma^{-1}_n-\sigma^{-1}_{n-1}>1$,
\item $\pi^{-1}_n=n$ iff $\sigma^{-1}_n>1$ and  $\sigma^{-1}_n-\sigma^{-1}_{n-1}=1$.
\end{enumerate}
\end{Pro}

\begin{figure}
\begin{center}
\begin{tabular}{c}
\begin{tabular}{c}
\unitlength=4mm
\begin{picture}(0.3,5)
\put(-0.7,2.3){$\scriptstyle\pi=$}
\end{picture}
\unitlength=1.2mm
\begin{picture}(19,19)
\put(0,0){\line(1,0){19}}
\put(0,0){\line(0,1){19}}
\put(19,0){\line(0,1){19}}
\put(0,19){\line(1,0){19}}
\put(2,17){\line(1,0){17}}
\put(2,0){\line(0,1){17}}
\put(1,18){\circle*{1.2}}
\put(9.2,8.5){$\scriptstyle\alpha$}
\end{picture} 

\unitlength=4mm
\begin{picture}(1,5)
\put(0.,2.3){$\rightarrow$}
\end{picture}
\unitlength=1.2mm
\begin{picture}(19,19)
\put(0,0){\line(1,0){19}}
\put(0,0){\line(0,1){19}}
\put(19,0){\line(0,1){19}}
\put(0,19){\line(1,0){19}}
\put(2,17){\line(1,0){17}}
\put(2,0){\line(0,1){17}}
\put(1,18){\circle*{1.2}}
\put(8.2,8.5){$\scriptstyle\psi(\alpha)$}
\end{picture} 

\unitlength=4mm
\begin{picture}(0.3,5)
\put(0,2.3){$\scriptstyle=\psi(\pi)$}
\end{picture}
\end{tabular}\\
(1)\\
\\

\begin{tabular}{c}
\unitlength=4mm
\begin{picture}(0.3,5)
\put(-0.7,2.3){$\scriptstyle\pi=$}
\end{picture}
\unitlength=1.2mm
\begin{picture}(19,19)
\put(0,0){\line(1,0){19}}
\put(0,0){\line(0,1){19}}
\put(19,0){\line(0,1){19}}
\put(0,19){\line(1,0){19}}

\put(0,11){\line(1,0){6}}
\put(0,17){\line(1,0){6}}
\put(6,11){\line(0,1){6}}

\put(8,0){\line(0,1){11}}
\put(8,11){\line(1,0){11}}

\put(2.2,13.5){$\scriptstyle\alpha$}
\put(12.6,5.2){$\scriptstyle\beta$}
\put(7,18){\circle*{1.2}} 
\end{picture} 
\unitlength=4mm
\begin{picture}(1,5)
\put(0.,2.3){$\rightarrow$}
\end{picture}
\unitlength=1.2mm
\begin{picture}(19,19)
\put(0,0){\line(1,0){19}}
\put(0,0){\line(0,1){19}}
\put(19,0){\line(0,1){19}}
\put(0,19){\line(1,0){19}}

\put(0,11){\line(1,0){6}}
\put(0,17){\line(1,0){6}}
\put(6,11){\line(0,1){6}}

\put(8,0){\line(0,1){11}}
\put(8,11){\line(1,0){11}}

\put(0.6,13.5){$\scriptstyle\psi(\alpha)$}
\put(11.5,5.2){$\scriptstyle\psi(\beta)$}
\put(7,18){\circle*{1.2}} 
\end{picture} 
\unitlength=4mm
\begin{picture}(0.7,5)
\put(0.0,2.3){$\scriptstyle=$}
\end{picture}
\unitlength=1.2mm
\begin{picture}(19,19)
\put(0,0){\line(1,0){19}}
\put(0,0){\line(0,1){19}}
\put(19,0){\line(0,1){19}}
\put(0,19){\line(1,0){19}}

\put(0,11){\line(1,0){6}}
\put(0,17){\line(1,0){6}}
\put(6,11){\line(0,1){6}}

\put(8,7){\line(1,0){4}}
\put(8,7){\line(0,1){4}}
\put(8,11){\line(1,0){4}}
\put(12,7){\line(0,1){4}}

\put(12,0){\line(0,1){5}}
\put(17,0){\line(0,1){5}}
\put(12,5){\line(1,0){5}}
\put(0.6,13.5){$\scriptstyle\psi(\alpha)$}
\put(9.3,8.7){$\scriptstyle\gamma$}
\put(14,1.8){$\scriptstyle\delta$}
\put(7,18){\circle*{1.2}} 
\put(18,6){\circle*{1.2}} 
\end{picture} 
\unitlength=4mm
\begin{picture}(1,5)
\put(0.,2.3){$\rightarrow$}
\end{picture}
\unitlength=1.2mm
\begin{picture}(19,19)
\put(0,0){\line(1,0){19}}
\put(0,0){\line(0,1){19}}
\put(19,0){\line(0,1){19}}
\put(0,19){\line(1,0){19}}

\put(0,11){\line(1,0){6}}
\put(0,17){\line(1,0){6}}
\put(6,11){\line(0,1){6}}

\put(6,4){\line(0,1){5}}
\put(6,4){\line(1,0){5}}
\put(6,9){\line(1,0){5}}
\put(11,4){\line(0,1){5}}

\put(15,0){\line(0,1){4}}
\put(15,4){\line(1,0){4}}
\put(0.6,13.5){$\scriptstyle\psi(\alpha)$}

\put(8,5.8){$\scriptstyle\delta$}
\put(16.4,1.7){$\scriptstyle\gamma$}
\put(12,10){\circle*{1.2}} 
\put(14,18){\circle*{1.2}} 
\end{picture} 
\unitlength=4mm
\begin{picture}(0.3,5)
\put(0,2.3){$\scriptstyle=\psi(\pi)$}
\end{picture}
\end{tabular}
\\
(2)\\
\\
\begin{tabular}{c}
\unitlength=4mm
\begin{picture}(0.3,5)
\put(-0.7,2.3){$\scriptstyle\pi=$}
\end{picture}
\unitlength=1.2mm
\begin{picture}(19,19)
\put(0,0){\line(1,0){19}}
\put(0,0){\line(0,1){19}}
\put(19,0){\line(0,1){19}}
\put(0,19){\line(1,0){19}}
\put(0,17){\line(1,0){17}}
\put(17,0){\line(0,1){17}}
\put(18,18){\circle*{1.2}}
\put(8.2,8.5){$\scriptstyle\alpha$}
\end{picture} 

\unitlength=4mm
\begin{picture}(1,5)
\put(0.,2.3){$\rightarrow$}
\end{picture}

\unitlength=1.2mm
\begin{picture}(19,19)
\put(0,0){\line(1,0){19}}
\put(0,0){\line(0,1){19}}
\put(19,0){\line(0,1){19}}
\put(0,19){\line(1,0){19}}
\put(0,17){\line(1,0){17}}
\put(17,0){\line(0,1){17}}
\put(18,18){\circle*{1.2}}
\put(7.0,8.5){$\scriptstyle\psi(\alpha)$}
\end{picture} 

\unitlength=4mm
\begin{picture}(0.7,5)
\put(0.0,2.3){$\scriptstyle=$}
\end{picture}

\unitlength=1.2mm
\begin{picture}(19,19)
\put(0,0){\line(1,0){19}}
\put(0,0){\line(0,1){19}}
\put(19,0){\line(0,1){19}}
\put(0,19){\line(1,0){19}}

\put(0,6){\line(1,0){9}}
\put(0,15){\line(1,0){9}}
\put(9,6){\line(0,1){9}}

\put(11,0){\line(0,1){6}}
\put(11,6){\line(1,0){6}}
\put(17,0){\line(0,1){6}}
\put(10,16){\circle*{1.2}}
\put(18,18){\circle*{1.2}}
\put(3.5,9.3){$\scriptstyle\gamma$}
\put(13.5,2.2){$\scriptstyle\delta$}
\end{picture} 

\unitlength=4mm
\begin{picture}(1,5)
\put(0.,2.3){$\rightarrow$}
\end{picture}

\unitlength=1.2mm
\begin{picture}(19,19)
\put(0,0){\line(1,0){19}}
\put(0,0){\line(0,1){19}}
\put(19,0){\line(0,1){19}}
\put(0,19){\line(1,0){19}}

\put(0,6){\line(1,0){9}}
\put(0,15){\line(1,0){9}}
\put(9,6){\line(0,1){9}}

\put(13,0){\line(0,1){6}}
\put(13,6){\line(1,0){6}}
\put(10,16){\circle*{1.2}}
\put(12,18){\circle*{1.2}}
\put(3.5,9.3){$\scriptstyle\gamma$}
\put(15.5,2.2){$\scriptstyle\delta$}
\end{picture} 
\unitlength=4mm
\begin{picture}(0.3,5)
\put(0,2.3){$\scriptstyle=\psi(\pi)$}
\end{picture}
\end{tabular}
\\
(3)\\
\end{tabular}
\end{center}
\caption{
\label{three_cases}
The three cases occurring in the definition of $\psi$: 
(1) $\pi^{-1}_n=1$, (2) $2\leq \pi^{-1}_n\leq n-1$, and (3) $\pi^{-1}_n=n$.}
\end{figure}
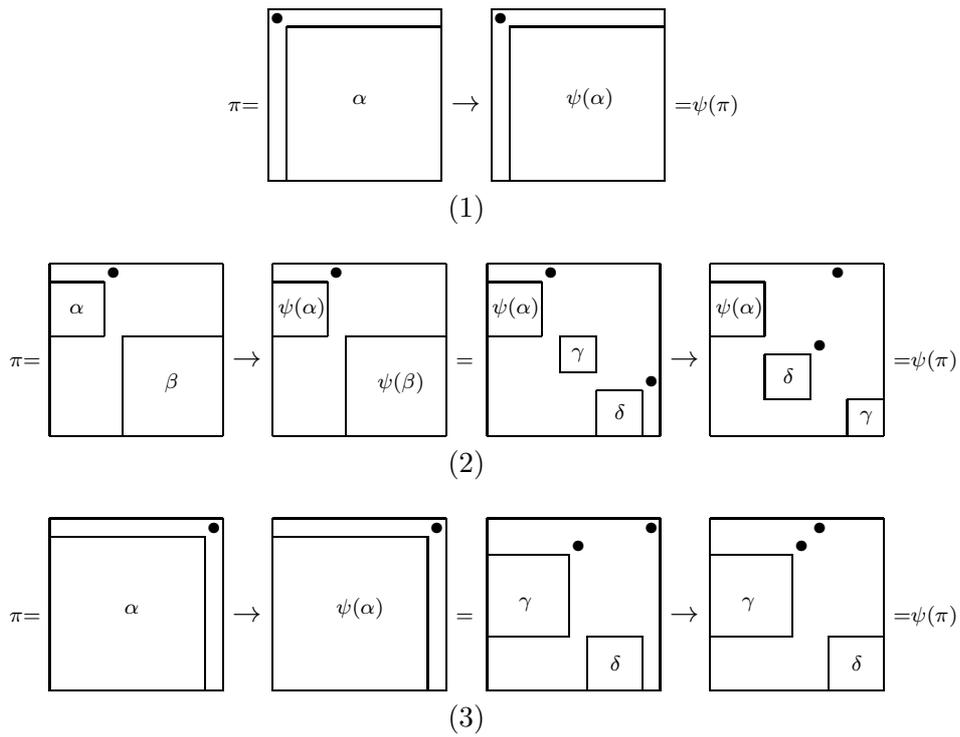
    
\begin{The}
The mapping $\psi$ is a bijection on $S_n(132)$, and for any $\pi\in S_n(132)$,
we have
$$
(\bSac,\des)\,\psi(\pi)=(\bSca,\des)\,\pi.
$$
\label{The_psi}
\end{The}
\begin{proof}
Let $\pi\in S_n(132)$. 
By the construction of $\psi$ and iteratively applying Fact~\ref{fact_decomposition} we have that 
$\psi(\pi)\in S_n(132)$.
And by induction on $n$, from Proposition \ref{Pro_psi} 
it follows that $\psi$ is injective and thus bijective; and
from Fact  \ref{fact_des} it follows that $\des\,\pi=\des\,\sigma$.

Now we show by induction on $n$ that 
$(\bSac)\,\psi(\pi)=(\bSca)\,\pi$, for any $\pi\in S_n(132)$, $n\geq 1$.

Clearly, for $n=1$, 
$(\bSac)\,\psi(\pi)=(\bSca)\,\pi$, and
let suppose that $(\bSac)\,\psi(\pi)=(\bSca)\,\pi$ for any $\pi\in S_n$ and 
$n<m$, and we will prove it for $\pi\in S_m$.      
\begin{enumerate}
\item If $\pi^{-1}_m=1$, then $\pi=1\ominus\alpha$ for some $\alpha\in S_{m-1}(132)$,
      and by definition, $\psi(\pi)=1\ominus\psi(\alpha)$. By the induction hypothesis
      we have 
      $(\bSac)\,\psi(\alpha)=(\bSca)\,\alpha$,
      and thus 
      $(\bSac)\,\psi(\pi)=(\bSac)\,\psi(\alpha)=
      (\bSca)\,\alpha=
      (\bSca)\,\pi$.             
\item If $1\leq \pi^{-1}_m\leq m-1$, let $\alpha$, $\beta$, $\gamma$ and $\delta$ 
      be the permutations appearing in the second case of the definition of $\psi$, and we have 
\begin{equation*}
\begin{array}{rclr}
(\bSac)\,\psi(\pi)
  & = & (\bSac)\, \psi(\alpha)+|\alpha|+
        (\bSac)\, \delta\oplus 1+
        (\bSac)\,\gamma \\
  & = & (\bSac)\, \psi(\alpha)+|\alpha|+
        (\bSac)\, \psi(\beta)  \\
  & = & (\bSca)\, \alpha+|\alpha|+
        (\bSca)\, \beta & \mbox{(by the induction hypothesis)}\\
  & = &	(\bSca)\, \pi.     
\end{array}
\end{equation*}    
      
\item If $\pi^{-1}_m=m$, let $\alpha$, $\gamma$ and $\delta$ 
      be the permutations appearing in the third case of the definition of $\psi$.
      Again,  
      $(\bSac)\,\psi(\alpha)=(\bSca)\,\alpha$, and
      
\begin{equation*}
\begin{array}{rclr}
(\bSac)\,\psi(\pi) 
  & = & (\bSac)\, \gamma\oplus1+
        (\bSac)\, \delta\\
  & = & (\bSac)\, \psi(\alpha)\\
  & = &	(\bSca)\, \alpha\\	
  & = &	(\bSca)\, \pi.     
\end{array}
\end{equation*}

\end{enumerate}
\end{proof}

\medskip

\noindent
{\bf 3.3 Equidistribution of $(\baSc,\des,\abF)$ and 
$(\bSac,\des,\abF)$ on $S_n(132)$: bijection~$\mu$
}


\medskip
\noindent
Based on the previously defined bijection $\psi$ we give a mapping $\mu$ on $S_n(132)$
and show that it is a bijection on $S_n(132)$, and Theorem \ref{The_bij_mu} proves the desired
equidistribution.

\medskip

Expressing in two different ways the major index of
a permutation, Lemma 2 in \cite{Vaj_13} (see also Corollary 14 in \cite{Vaj_11})
shows that for any permutation $\pi$ (not necessarily in $S_n(132)$) we have 

\begin{equation*}
(\bSac+\baF)\, \pi=
(\bSca +\ba)\, \pi.
\label{Le_2}
\end{equation*}

Actually, $(\ba)\,\pi$ is equal to $\des\,\pi$, and
the above relation becomes

\begin{equation}
(\bSac+\baF)\, \pi=
(\bSca+\des)\, \pi,
\label{Lemma_2_res}
\end{equation}
and the next lemma follows.

\begin{Le}
\label{lem_vv}
For any permutation $\pi$ we have
$$
(\bSac)\, \pi\oplus 1= (\bSca+\des)\,\pi.
$$

\end{Le}
\proof
From relation (\ref{Lemma_2_res}) we have 
$$
(\bSac+\baF)\, \pi\oplus 1=
(\bSca+\des)\, \pi\oplus 1,
$$
and the statement holds by considering that 
$(\baF)\, \pi\oplus 1=0$, $\des\, \pi\oplus 1=\des\,\pi$, 
and $(\bSca)\, \pi\oplus 1=(\bSca)\, \pi$.
\endproof

In the proof of the following theorem we will use the next 
easy to understand fact.

\begin{Fact} For any permutation $\pi$, we have 
$(\baSc)\, \pi\oplus 1=(\baSc+\des)\,\pi$.
\label{fact_1}
\end{Fact}

\medskip

The mapping $\mu$ on $S_n(132)$ is recursively defined as:
if $\pi$ is the empty permutation, then $\mu(\pi)=\pi$; 
and if $\pi\in S_n(132)$, $n\geq 1$, with $\pi=\alpha \ominus (\beta \oplus 1)$
for some $132$-voiding permutations $\alpha$ and $\beta$,
then  
$$
\mu(\pi)=\mu(\alpha)\ominus (\mu(\psi(\beta))\oplus1),
$$
where
$\psi$ is the bijection define in Subsection 3.2.
See Figure \ref{cons_mu} for this recursive construction.

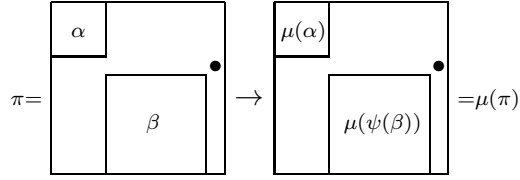
\begin{figure}[h]
\begin{center}
\begin{tabular}{c}
\unitlength=4mm
\begin{picture}(0.3,5)
\put(-0.7,2.3){$\scriptstyle\pi=$}
\end{picture}
\unitlength=1.2mm
\begin{picture}(19,19)
\put(0,0){\line(1,0){19}}
\put(0,0){\line(0,1){19}}
\put(19,0){\line(0,1){19}}
\put(0,19){\line(1,0){19}}

\put(0,13){\line(1,0){6}}
\put(6,13){\line(0,1){6}}

\put(6,0){\line(0,1){11}}
\put(6,11){\line(1,0){11}}
\put(17,0){\line(0,1){11}}

\put(2.2,15.2){$\scriptstyle\alpha$}
\put(10.4,5.){$\scriptstyle\beta$}
\put(18,12){\circle*{1.2}} 
\end{picture} 
\unitlength=4mm
\begin{picture}(1,5)
\put(0.,2.3){$\rightarrow$}
\end{picture}
\unitlength=1.2mm
\begin{picture}(19,19)
\put(0,0){\line(1,0){19}}
\put(0,0){\line(0,1){19}}
\put(19,0){\line(0,1){19}}
\put(0,19){\line(1,0){19}}

\put(0,13){\line(1,0){6}}
\put(6,13){\line(0,1){6}}

\put(6,0){\line(0,1){11}}
\put(6,11){\line(1,0){11}}
\put(17,0){\line(0,1){11}}

\put(0.7,15.2){$\scriptstyle\mu(\alpha)$}
\put(7.6,5.){$\scriptstyle\mu(\psi(\beta))$}
\put(18,12){\circle*{1.2}} 
\end{picture} 
\unitlength=4mm
\begin{picture}(0.3,5)
\put(0,2.3){$\scriptstyle=\mu(\pi)$}
\end{picture}
\end{tabular}
\end{center}
\caption{
\label{cons_mu}
The recursive definition of $\mu(\pi)$.}
\end{figure}
 
Since $\psi$ is a bijection on $S_n(132)$, $n\geq 0$, it follows that $\mu(\pi)$ avoids 
$132$ whenever $\pi$ does so, and thus $\mu(\pi)\in S_n(132)$ for any 
$\pi\in S_n(132)$. With the notations above, it is clear that 
$(\abF)\,\mu(\pi)=(\abF)\,\pi=|\beta|$ and considering again the bijectivity of 
$\psi$,
by induction on $n$ it follows that $\mu$ is injective, and thus bijective.

\begin{The}
If $\pi\in S_n(132)$, then
$$
(\baSc,\des,\abF)\,\mu(\pi)=(\bSac,\des,\abF)\,\pi.
$$
\label{The_bij_mu}
\end{The}
\proof 
Clearly, $(\abF)\,\mu(\pi)=(\abF)\,\pi$,
and the remaining of the proof is by induction on $n$.
Obviously, the statement holds for $n=0$, and let suppose that it is true for any $n<m$, 
and consider $\pi=\alpha\ominus (\beta\oplus1)\in S_m(132)$ for
some $132$-avoiding permutations $\alpha$ and $\beta$.

The bijection $\mu$ preserves $\des$ statistic. Indeed, 
(using the Iverson bracket notation) considering 
$[|\alpha|\neq 0]$ equal to $0$ (resp. $1$) if $\alpha$ is empty (resp. not empty)
we have
\begin{equation*}
\begin{array}{rclr}
\des\, \mu(\pi)& = & \des\, \mu(\alpha)\ominus (\mu(\psi(\beta))\oplus1) \\
               & = & \des\, \mu(\alpha) + \des\, \mu(\psi(\beta))+[|\alpha|\neq 0]   \\
	       & = & \des\, \alpha + \des\, \psi(\beta)+[|\alpha|\neq 0]    &\mbox{(by the induction hypothesis)}\\
               & = & \des\, \alpha + \des\, \beta+[|\alpha|\neq 0]    &\mbox{(since }\psi \mbox{ preserves } \des \mbox{)}\\
               & = & \des\, \alpha \ominus (\beta \oplus 1)\\ 
               & = & \des\, \pi. 
\end{array}
\end{equation*}

Finally, we show that $(\baSc)\,\mu(\pi)=(\bSac)\,\pi$.

\begin{equation*}
\begin{array}{rclr}
(\baSc)\,\mu(\pi) & = & (\baSc)\,\mu(\alpha)\ominus (\mu(\psi(\beta))\oplus1) & \\
                            & = & (\baSc)\,\mu(\alpha) + (\baSc)\, \mu(\psi(\beta))\oplus1 &\\
                            & = & (\baSc)\,\mu(\alpha) + (\baSc)\, \mu(\psi(\beta)) + \des\,\mu(\psi(\beta)) &
   \mbox{(by Fact \ref{fact_1})}\\
                            & = & (\bSac)\,\alpha + (\bSac)\, \psi(\beta) + \des\,\beta &
   \mbox{(by the induction hypothesis}\\
                            & & & \mbox{and $\mu$ and $\psi$ preserve $\des$)}\\
                            & = & (\bSac)\,\alpha + (\bSca)\,\beta + \des\,\beta &  \mbox{(by Theorem \ref{The_psi})}\\
                            & = & (\bSac)\,\alpha + (\bSac)\,\beta \oplus 1 &
			   \mbox{(by Lemma \ref{lem_vv}})\\
                            & = & (\bSac)\,\alpha \ominus (\beta \oplus 1)&\\
                            & = & (\bSac)\,\pi.&
\end{array}
\end{equation*}

\endproof

\medskip

\noindent
{\bf 3.4 Equidistribution of $(\bcSa,\cSab,\des)$ and $(\cSab,\bcSa,\des)$}


\medskip
\noindent
It is easy to see that the inverse of a permutation (defined at the end of Section \ref{Notations})
satisfies: if $\pi=\alpha \ominus (\beta \oplus 1)$, then 
$\pi^{-1}=(\beta^{-1} \oplus 1)\ominus\alpha^{-1}$, see Figure \ref{cons_inv}.

As mentioned at the end of Section \ref{Notations}, the inverse of a vincular pattern
is not longer a vincular pattern, however we have the following.

\begin{Pro}
\label{Pro_1}
For any $\pi\in S_n(132)$, we have
$$
(\bcSa,\cSab,\des)\,\pi^{-1}=(\cSab,\bcSa,\des)\,\pi.
$$
\end{Pro}
\proof
Trivially, the statement holds for $n=0$, and let suppose that it is true for any $n<m$, 
and consider $\pi=\alpha \ominus (\beta \oplus 1)\in S_m(132)$ for
some $132$-avoiding permutations $\alpha$ and $\beta$. 

If $\alpha$ is empty, then $\des\,\pi^{-1}=\des\,\beta^{-1}$ and
$\des\,\pi=\des\,\beta$; 
otherwise,
$\des\,\pi^{-1}=\des\,\beta^{-1}+\des\,\alpha^{-1}+1$ and
$\des\,\pi=\des\,\alpha+\des\,\beta+1$.
In both cases, by the induction hypothesis it follows that $\des\,\pi^{-1}=\des\,\pi$.

An occurrence of $\bcSa$ in $\pi^{-1}$ can be found  either 
in $\beta^{-1}$, or in $\alpha^{-1}$, or when $\beta$ is not empty, 
has the form $abc$ with $a$ the last symbol of $\beta^{-1}$, $b=m$
(the largest symbol of $\pi$) and 
$c$ a symbol of $\alpha^{-1}$.
Similarly, an occurrence of $\cSab$ in $\pi$ can be found either 
in $\alpha$, or in $\beta$, or when $\beta$ is not empty, 
has the form $abc$ with $a$ a symbol of $\alpha$, $b$ 
the last symbol of $\beta$ and $c$ the last symbol of $\pi$.

Thus, when $\beta$ is not empty
$$
(\bcSa)\,\pi^{-1}=(\bcSa)\,\beta^{-1}+(\bcSa)\,\alpha^{-1}+|\alpha^{-1}|
$$
and 
$$
(\cSab)\,\pi=(\cSab)\,\alpha+(\cSab)\,\beta +|\alpha|,
$$
and by the induction hypothesis it follows that 
$(\bcSa)\,\pi^{-1}=(\cSab)\,\pi$.
And when $\beta$ is empty
$$
(\bcSa)\,\pi^{-1}=(\bcSa)\,\alpha^{-1}
$$
and 
$$
(\cSab)\,\pi=(\cSab)\,\alpha,
$$
and again $(\bcSa)\,\pi^{-1}=(\cSab)\,\pi$.

Since $^{-1}$ is an involution on $S_n(132)$, it follows that
$(\bcSa)\,\pi^{-1}=(\cSab)\,\pi$, and the statement holds.
\endproof

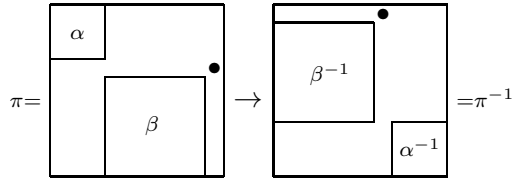
\begin{figure}[h]
\begin{center}
\begin{tabular}{c}
\unitlength=4mm
\begin{picture}(0.3,5)
\put(-0.7,2.3){$\scriptstyle\pi=$}
\end{picture}
\unitlength=1.2mm
\begin{picture}(19,19)
\put(0,0){\line(1,0){19}}
\put(0,0){\line(0,1){19}}
\put(19,0){\line(0,1){19}}
\put(0,19){\line(1,0){19}}

\put(0,13){\line(1,0){6}}
\put(6,13){\line(0,1){6}}

\put(6,0){\line(0,1){11}}
\put(6,11){\line(1,0){11}}
\put(17,0){\line(0,1){11}}

\put(2.2,15.2){$\scriptstyle\alpha$}
\put(10.4,5.){$\scriptstyle\beta$}
\put(18,12){\circle*{1.2}} 
\end{picture} 
\unitlength=4mm
\begin{picture}(1,5)
\put(0.,2.3){$\rightarrow$}
\end{picture}
\unitlength=1.2mm
\begin{picture}(19,19)
\put(0,0){\line(1,0){19}}
\put(0,0){\line(0,1){19}}
\put(19,0){\line(0,1){19}}
\put(0,19){\line(1,0){19}}

\put(0,6){\line(1,0){11}}
\put(11,6){\line(0,1){11}}
\put(0,17){\line(1,0){11}}

\put(13,0){\line(0,1){6}}
\put(13,6){\line(1,0){6}}

\put(4.,10.5){$\scriptstyle\beta^{-1}$}
\put(13.8,2.2){$\scriptstyle\alpha^{-1}$}
\put(12,18){\circle*{1.2}} 
\end{picture} 
\unitlength=4mm
\begin{picture}(0.3,5)
\put(0,2.3){$\scriptstyle=\pi^{-1}$}
\end{picture}
\end{tabular}
\end{center}
\caption{
\label{cons_inv}
The recursive construction of $\pi^{-1}$, for $\pi\in S(132)$.}
\end{figure}

\section{Conclusions}

We showed bijectively the joint equidistribution on the set $S_n(132)$ of 
$132$-voiding permutations of some length three vincular patterns together with other statistics. 
In particular, for the sets of vincular patterns $\{\bSca,\bSac,\baSc\}$ and
$\{\bcSa,\cSab\}$, we showed that
the patterns within each set  are equidistributed 
on $S_n(132)$.
By applying permutation symmetries, other similar results can be derived.
For instance, from the equidistribution of $\baSc$ and $\bSac$ on $S_n(132)$ 
(belonging to the first set, see Subsection 3.3) it follows, by applying

\begin{itemize}
\item the reverse operation, the equidistribution of $\cSab$ and $\caSb$
      on $S_n(231)$,
\item the complement operation, the equidistribution of $\bcSa$ and $\bSca$
      on $S_n(312)$, and
\item the complement and the reverse operations
      (in any order), the equidistribution of $\aScb$ and $\acSb$
      on $S_n(213)$.
\end{itemize}
Moreover, computer experiments show that, up to these two symmetries, the
patterns in 
$\{\bSca,\bSac,\baSc\}$ and those in
$\{\bcSa,\cSab\}$ 
are the only length three proper (not classical nor adjacent) 
vincular patterns 
which are equidistributed on a set of permutations avoiding a classical  length three
pattern.

\end{document}